\documentclass[12pt,twoside]{article}
\usepackage{amsmath,amssymb,amsfonts,amsthm,epsfig,color,graphicx,eucal,mathrsfs}
\usepackage{latexsym}
\usepackage{graphics}
\usepackage{verbatim}
\usepackage{my_commands}
\usepackage{asycolors}
\usepackage{rotating}

\newtheorem{lemma}{Lemma}[section]
\newtheorem{proposition}{Proposition}[section]
\newtheorem{theorem}{Theorem}[section]
\newtheorem{definition}{{Definition}}[section]

\title{Constructive Approach of the Solution of Riemann
  Problem for Shallow Water Equations with Topography and
  Vegetation}

\author{Stelian Ion$^*$, Dorin Marinescu$^*$, Stefan-Gicu
  Cruceanu\footnote{`Gheorghe Mihoc-Caius Iacob'' Institute
    of Mathematical Statistics and Applied Mathematics,
    Romanian Academy, 050711 Bucharest, Romania, {\tt emails:
      ro\_diff@yahoo.com, marinescu.dorin@ismma.ro,
      gcruceanu@ismma.ro}.  \protect\newline This work was
    partially supported by a grant of the Ministry of
    Research and Innovation, CCCDI-UEFISCDI, project number
    PN-III-P1-1.2-PCCDI-2017-0721/34PCCDI/2018, within PNCDI
    III.}}

\date{}

\begin{document}
\maketitle

\begin{abstract}
  We investigate the Riemann Problem for a shallow water
  model with porosity and terrain data.  Based on recent
  results on the local existence, we build the solution in
  the large settings (the magnitude of the jump in the
  initial data is not supposed to be ``small enough'').  One
  difficulty for the extended solution arises from the
  double degeneracy of the hyperbolic system describing the
  model.  Another difficulty is given by the fact that the
  construction of the solution assumes solving an equation
  which has no global solution.  Finally, we present some
  cases to illustrate the existence and non-existence of the
  solution.
\end{abstract}

\def\jumpp#1{\left|\left|#1\right|\right|}

\def\grad#1{{\rm grad}\,#1}
\def\diffp#1#2{\displaystyle\frac{\partial #1}{\partial #2}}
\def\jump#1{\left[\left|#1\right|\right]}
\def\jumpp#1{\left\lfloor\left|#1\right|\right\rfloor}

\section{Introduction}
Riemann Problem is a classical topic in the theory of
hyperbolic systems, \cite{lax, glimm, lefloch, smoller} and
it is widely used in testing or elaborating numerical
schemes, \cite{toro}, to refer to just a few classical
papers.  The study of Riemann Problem for non-conservative
hyperbolic systems requires a new concept on the definition
of the discontinuous solution.  In this sense, two main
ingredients are introduced: {\it measure solution},
\cite{dalmaso} and {\it path connection},
\cite{path1-lefloch, path2-lefloch, volpert}.  Using a path
connection, one can define the Rankine-Hugoniot relations
that relate the two side values of a shock solution on the
discontinuity curve.  The shock solutions depend on the path
connection.

The shallow water equations with topography and vegetation
is a widely used mathematical model in environmental
sciences to study the flow of water into natural systems,
\cite{delestre, sdsai_ovidius, nepf}.  The model fits into
the class of non-conservative hyperbolic systems, where
there are known several formulations for jump relations,
\cite{cozzolin1, cozzolin2, guinot, leflochrp}.

In \cite{sds_ovidius}, we have introduced the jump relations
using a class of path connections that was chosen on the
basis of certain physical arguments.

In this paper, we investigate the existence of the solution
in the ``large'': the Riemann Problem data are not
restricted to be ``closed enough''.  The ``constructive''
word from the title must be interpreted as follows.  First,
a problem being given, there is a way to affirm that the
problem has or has not a solution.  In case of an
affirmative answer, there is an algorithm that allows one to
build the solution.  Secondly, we do not have results that
can give general conditions for the existence of the
solution.

For an easier understanding of the results, we briefly
recall the shallow water model with topography and
vegetation.  For more details, the readers are referred to
the papers \cite{cozzolin1, cozzolin2, guinot, sds_ovidius}.

In the absence of the friction terms and if there are not
water gain or loss, the 1D shallow water equations with
topography and vegetation can be written as
\begin{equation}
  \label{swe.01}
  \begin{split}
    \displaystyle\frac{\partial}{\partial t}\theta h+\partial_x\left(\theta h u\right)&=0,\\
    \diffp{}{t}\theta hu+\partial_x \left(\theta hu^2\right)+\theta h\partial_x w&=0,
  \end{split}
\end{equation}
where $h(t,{x})$ is the water height, $u(t,{x})$ the water
velocity and $z(x)$ the soil surface level.  The function
$w=g(z+h)$ stands for the potential of the water level and
$g$ is gravitational acceleration.  The variation of the
cover plant density is taken into account through the
function $\theta({x})$, the porosity of the plant cover.

The Riemann problem for the shallow water equations with
topography and vegetation consists in finding a solution in
the class of functions with bounded variation for the
equations (\ref{swe.01}) with the following initial
conditions
\begin{equation}
  \label{srp-02}
  (h,u)_{t=0}=\left\{
    \begin{array}{ll}
      (h^L, u^L),&x<0,\\
      (h^R, u^R),&x>0.
    \end{array}
  \right.
\end{equation}
The terrain data (the porosity $\theta$ and the soil
surface $z$) are defined by
\begin{equation}
  \label{srp-03}
  (\theta, z)=\left\{
    \begin{array}{ll}
      (\theta^L, z^L),&x<0,\\
      (\theta^R, z^R),&x>0.
    \end{array}
  \right.
\end{equation}

\section{The Riemann Problem for arbitrary data}

A solution of the problem is built  by using rarefaction
waves and shock waves.  The rarefaction waves are smooth
solutions of (\ref{swe.01}) in a domain where the terrain
functions are constant.

For a function $\Psi$, $\jump{\Psi}$ stands for the jump
$\Psi^{+}-\Psi^{-}$.
The shock wave solutions verify the classical
Rankine-Hugoniot relations:
\begin{equation}
  \label{swe.08}
  \begin{split}
    -\sigma\jump{h}+\jump{hu}&=0,\\
    -\sigma\jump{hu}+\jump{hu^2+g{h^2}/{2}}&=0,\\
  \end{split}
\end{equation}
in the domains $x<0$ or $x>0$ or generalized Rankine-Hugoniot relations
\begin{equation}
  \label{swe.10}
  \begin{split}
    \jump{\theta hu}&=0,\\
    \jump{\theta hu^2}+g
    \displaystyle\int\limits_0^1\theta(s)h(s)\displaystyle\frac{d(z(s)+h(s))}{{\rm d}s}{\rm d}s&=0,\\
  \end{split}
\end{equation}
for a steady shock located at $x=0$.  The integral is
evaluated on a path connection curve
$\{\theta(s; \theta^L, \theta^R), z(s; z^L, z^R), h(s; h^L,
h^R)\}$.

\subsection{Riemann Constructor. $(z, \theta)$ constant function} 
Whenever the terrain function are constant the Riemann
problem can be solved by using the two kind of the waves,
rarefaction waves and shock waves.  In the phase space
$(h,u)$ one defines a 1-wave curve {\it issuing from a point
  $(h_L,u_L)$}
\[
W_1(h; h_L,u_L):=\{(h,u_1(h;h_L,u_L))|h>0\}   
\]
where  
\begin{equation}
  \label{rpl_eq.1}
  u_1(h;h_L,u_L)=
  \left\{
    \begin{array}{ll}
      u_L+2\sqrt{gh_L}\left(1-\sqrt{\displaystyle\frac{h}{h_L}}\right),&h<h_L,\\
      u_L+\sqrt{gh_L}\left(1-{\displaystyle\frac{h}{h_L}}\right)
      \sqrt{\displaystyle\frac{1}{2}\left(1+\displaystyle\frac{h_L}{h}
      \right)},&h>h_L 
    \end{array}
  \right.
\end{equation}
and 2-backward wave curve {\it reaching a point $(h_R,u_R)$} 
\[ W^B_2(h;h_R,u_R):=\{(h,u_2^B(h;h_R,u_R)|h>0\},\] where
\begin{equation}
  \label{rpl_eq.2}
  u^{B}_2(h;h_R,u_R)=
  \left\{
    \begin{array}{ll}
      u_R-2\sqrt{gh_R}\left(1-\sqrt{\displaystyle\frac{h}{h_R}}\right),&h<h_R,
      \\u_R-\sqrt{gh_R}\left(1-{\displaystyle\frac{h}{h_R}}\right)
      \sqrt{\displaystyle\frac{1}{2}\left(1+\displaystyle\frac{h_R}{h}\right)},&h>h_R. 
    \end{array}
  \right.
\end{equation}
The interpretation of the two wave curves is as follows: 

(a) $W_1(h; h_L,u_L)$.  A point $(h_L,u_L)$ being given as
the left state in the Riemann problem, the curve
$W_1(h; h_L,u_L)$ defines all right states that can be
connected to the left state either by a 1--shock wave,
$h>h_L$ or by a 1--rarefaction wave, $h<h_L$.

(b) $W^B_2(h;h_R,u_R)$.  A point $(h_R,u_R)$ being given as
the right state in the Riemann problem, the curve
$W^B_2(h; h_R,u_R)$ defines all left states that can be
connected to the right state either by a 2--shock wave,
$h>h_R$ or by a 2--rarefaction wave, $h<h_R$.

The shock speed on each curve can be calculated by formula
\begin{equation}
  \label{rpl_eq.3}
  \begin{array}{ll}
    \sigma_1(h; h_L, u_L)=u_L-\sqrt{gh}\sqrt{\displaystyle\frac{1}{2}\left(1+\displaystyle\frac{h}{h_L}\right)},&h>h_L,\\
    \sigma_2(h; h_R, u_R)=u_R+\sqrt{gh}\sqrt{\displaystyle\frac{1}{2}\left(1+\displaystyle\frac{h}{h_R}\right)},&h>h_R
  \end{array}
\end{equation}
and the eigenvalues are given by
\begin{equation}
  \label{rpl_eq.4}
  \begin{array}{ll}
    \lambda_1(h;h_L,u_L)=\lambda_1(h_L,u_L)+3(\sqrt{gh_L}-\sqrt{gh}),&h<h_L,\\
    \lambda_2(h;h_R,u_R)=\lambda_2(h_R,u_R)+3(\sqrt{gh}-\sqrt{gh_R}),&h<h_R.
  \end{array}      
\end{equation}

For the case of terrain constant functions, the solution of
Riemann Problem for arbitrary data is a composite wave that
can be found following two steps:

{\bf Riemann Constructor. $(z, \theta)$ Constant function}

\begin{enumerate}
\item[{\bf Step 1}] {\it Find the intersection point $h_*$
    such that,}
\[
W^B_2(h_*;h_R,u_R)=W_1(h_*;h_L,u_L).
\]  
\item[{\bf Step 2}] {\it The composite wave curve of the
    solution for the Riemann problem is}
\[
W^B_2(h_R;W_1(h_*;h_L,u_L)).
\]  
\end{enumerate}

\begin{figure}[t]
  \centering
  \begin{tabular}{cc}
    \includegraphics[width=0.49\linewidth]{RP_constructor_case1.png}
    & \includegraphics[width=0.49\linewidth]{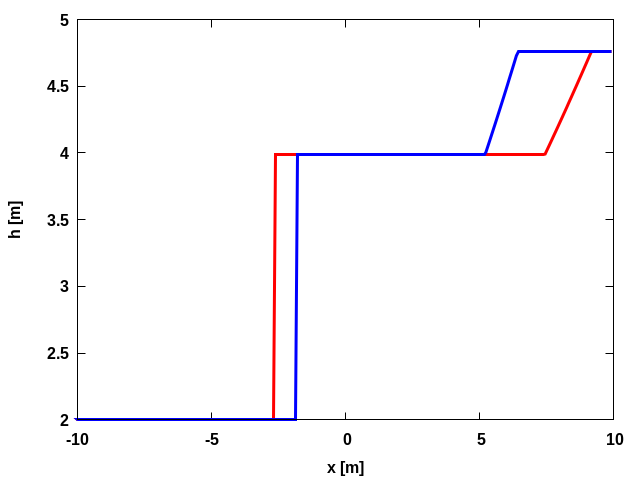}
  \end{tabular}
  \label{rp_fig.1}
  \caption{The figure illustrates an application of the
    construction algorithm to solve a Riemann Problem.  The
    left picture contains the 1-shock curve (green scatter
    plot), 1-rarefaction wave (green line), 2-backward shock
    (blue scatter plot) and 2-backward rarefaction wave
    (blue line).  The right picture contains the graphs of
    water height at two different moments of time
    $t=1{\rm s}$ (red line) and $t=0.7 {\rm s}$ (blue
    line).}
\end{figure}

\subsection{Riemann Constructor. Jump in $(z, \theta)$ }

For the case of jump in the terrain function, there are
three types of waves.  In addition to the first two waves
$W_1$, $W_2$, there is another steady shock wave $W_3$ that
results as a solution of the generalized Rankine-Hugoniot
equation, (\ref{swe.10}).  To build the third wave, it is
necessary to introduce a physical path that connects two
arbitrary states ${\cal P}^-:=(h^-,u^-,z^-,\theta^-)$ and
${\cal P}^+:=(h^+,u^+,z^+,\theta^+)$.  In the paper
(\cite{sds_ovidius}), we define a physical path by
\begin{equation*}
  \begin{split}
    h(s;h^-,h^+) & =h^-+\phi(s)(h^+-h^-),\\
    z(s;z^-,z^+) & =z^1+\phi(s)(z^+-z^-),\\
    \theta(s;\theta^-,\theta^+) & =\theta^-+
                                  \displaystyle\frac{\theta^-\phi(s)}{\theta^-\phi(s)+(1-\phi(s))\theta^+}(\theta^+-\theta^-),
  \end{split}
\end{equation*}
where $\phi(s)$ is an arbitrary smooth and monotone function
that satisfies $\phi(0)=0$ and $\phi(1)=1$.  Based on this
path, one can define the $W_3$--steady shock curve as
follows.

We introduce the notations
\[
\jumpp{z}=\frac{z^+-z^-}{h^-},\quad \theta=\frac{\theta^+}{\theta^-},\quad 
{\rm Fr}^2_-=\frac{(u^-)^2}{gh^-},
\]  
and the function
\[
\psi(y;\theta,\jumpp{z},{\rm Fr}_{-}):=-b(\theta)y^3-(a(\theta)-b(\theta)(1-\jumpp{z}))y^2+((1-\jumpp{z})a(\theta)-{\rm Fr}^2_-)y+
\frac{{\rm Fr}^2_-}{\theta},
\]
where $a(\theta)$ and $b(\theta)$ are given by
\[
b(\theta)=\frac{\theta(\theta-1-\theta\log{\theta})}{(\theta-1)^2}, \quad
  a(\theta)=-1-\frac{b(\theta)}{\theta}.
\]

\begin{definition}[3--wave]
  Given the terrain configuration $(z^-,z^+)$,
  $(\theta^-,\theta^+)$ and the left state $U^-=(h^-,u^-)$ a
  right state $U^+=(h^+,u^+)$ is defined by
  \begin{equation}
    \label{rpl_eq.6}
    \begin{split}
      h^+ & =hh^-,\\
      u^+ & =\displaystyle\frac{u^-}{\theta}\displaystyle\frac{1}{h},
    \end{split}
  \end{equation}
  where $h$ is the solution of the equation
  \begin{equation}
    \label{rpl_eq.7}
    \psi(x;\theta,\jumpp{z},{\rm Fr}_{-})=0
  \end{equation}
  that minimizes the function
  \begin{equation}
    \label{rpl_eq.8}
    {\cal E}(x)={\rm max}\{\left|\theta^{-1}-x\right|,\left|1-\jumpp{z}-x\right|\}.
  \end{equation}
  \label{def-3-wave}
\end{definition}

To understand the necessity of minimization criterion
required in Definition \ref{def-3-wave}, the following
remarks are in order.

Since the equation $\psi(x;\theta,\jumpp{z},{\rm Fr})=0$ can
have two positive solutions, it is necessary to introduce a
criterion to select a physically admissible solution.

When one solves the local problem, the selection of a
solution is based on the continuity argument, in the sense
that if the ratio $\theta$ and the soil surface jump
approach $1$ and $0$, respectively, then the left and right
states must be equal to each other.  In computations, one
can use as selection criterion the comparison of the left
Froude number with the unity, (see Theorem 2.1 in
\cite{sds_ovidius}).

But when we deal with large data, there is no guarantee that
selection based on the Froude number determines a solution
that goes to unity when the terrain data become continuous.

Here is an example.  Assume that $1-\jumpp{z}=1/\theta$.  In
such a case, there are two positive solutions:
\[
h_1=\displaystyle\frac{1}{\theta}, \quad
h_2=\displaystyle\frac{a(\theta)+\sqrt{a(\theta)^2-4b(\theta){\rm
      Fr}^2}}{-2b(\theta)}.
\]
Using the identity $a(\theta)+b(\theta)/\theta=-1$, one can
prove that if ${\rm Fr}^2<1/\theta$ then $h_2<h_1$.
Assuming that $\jumpp{z}<0$ and $1<{\rm Fr}^2<1/\theta$,
then the physical solution is $\beta=h_2$.  But
\[
\lim\limits_{\theta\rightarrow 1}h_2=\frac{-1+\sqrt{1+8{\rm Fr}^2}}{2}\neq 1
\]
On the other hand, the solution $h_1$ is obtained by a
continuous deformation of the solution $h_1=1$.

To overcome this problem, we introduce as a selection
criterion the minimization of the function ${\cal E}(x)$
that is a measure of the magnitude of the discontinuity in
free surface and fluid velocity.

There are two major difficulties encountered when building
the solution of the problem:

(a) the terrain jump equation can be unsolvable;

(b) along the wave curves $W_1$ and $W_2$, the eigenvalues
$\lambda_i$ change their sign, which implies that two
different states situated on the same wave curve but having
the eigenvalues with opposed sign can not be connected due
to the presence of the discontinuity line ($x=0$) that
separates the left and right states.

\begin{table}[t]
  \begin{tabular}{l|ccccc}
    ${\rm Fr}^2$&0&&$1/\theta+(a(\theta)+2)(1/\theta-(1-\jumpp{z}))$&&$\infty$\\\hline
    $\psi_2({\rm Fr})$&$-$&$\nearrow$&$1/\theta-(1-\jumpp{z})$&$\searrow$&$-$\\\hline
    $\partial_{{\rm Fr}^2}\psi_2$&\multicolumn{2}{c}{$+$}&$0$&\multicolumn{2}{c}{$-$}\\
      & \multicolumn{2}{c}{$\widetilde{h}_2({\rm Fr})<1/\theta $}&$\widetilde{h}_2({\rm Fr})=1/\theta$&\multicolumn{2}{c}{$\widetilde{h}_2({\rm Fr})>1/\theta$}\\\hline
  \end{tabular}
  \caption{The influence of the Froude number on the
    existence of the solution for the terrain jump equation.} 
\end{table}

The behavior of the solutions of the equation
$\psi(h,\theta, \jumpp{z},{\rm Fr})=0$ as functions of
${\rm Fr}$ is analyzed in the next lemma.  If
$\theta<1-\jumpp{z}$, there always are two positive
solutions for the terrain jump equation.  But in the case
$\theta>1-\jumpp{z}$, there are two critical numbers
${\rm Fr}_{*}$ and ${\rm Fr}^{*}$ such that if
$({\rm Fr}_{*})^2<{\rm Fr}^2< ({\rm Fr}^{*})^2$, then the
solution does not exist.
\begin{lemma} 
  \label{lemma-RP-01}
  Let $\theta$ and $\jumpp{z}$ be fixed. Let ${\rm Fr}$ be a
  variable parameter in the terrain jump equation
  $\psi(h; \theta,\jumpp{z},{\rm Fr})=0$.  Then, we can
  state the following properties about the solutions:

  {\rm (a)} If $1/\theta<1-\jumpp{z}$, then the terrain jump
  equation has two solutions $h_1<h_2$, for any Froude
  number, and
  \begin{equation}
    \label{RP_02}
    h_1<\frac{1}{\theta}<1-\jumpp{z}<h_2.
  \end{equation}

  {\rm (b)} If $1/\theta>1-\jumpp{z}$, then there are two
  critical Froude numbers ${\rm Fr}_{*}$ and ${\rm Fr}^{*}$,
  such that\\
  (i) if
  ${\rm Fr}^2\in \left(0, ({\rm Fr}_{*})^2\right)\cup
  \left(({\rm Fr}^{*})^2,\infty\right)$,
  then there are two solutions that satisfy the inequalities
  \begin{equation}
    \label{RP_03}
    \begin{array}{ll}
      h_1<h_2<1-\jumpp{z}<\displaystyle\frac{1}{\theta},
      &\quad {\rm if\,\,} {\rm Fr}^2<({\rm Fr}_{*})^2,\\
      1-\jumpp{z}<\displaystyle\frac{1}{\theta}<h_1<h_2,
      &\quad{\rm if\,\,} {\rm Fr}^2>({\rm Fr}^{*})^2;
    \end{array}
  \end{equation}
  (ii) if
  ${\rm Fr}^2\in \left(({\rm Fr}_{*})^2,({\rm
      Fr}^{*})^2\right)$, then there are no solutions.
\end{lemma}

\begin{proof}
  The inequalities (\ref{RP_02}) are consequence of the
  property that if $1/\theta<1-\jumpp{z}$, then
  $\psi(1/\theta;\theta,\jumpp{z},{\rm Fr})<0$ and
  $\psi(1-\jumpp{z};\theta,\jumpp{z},{\rm Fr})<0$.

  To prove (\ref{RP_03}), we analyze the behavior of the
  minimum value of $\psi(h; \theta,\jumpp{z},{\rm Fr})$.
  Let $\widetilde{h}_2(\theta,\jumpp{z},{\rm Fr})$ be the
  positive solution of the equation
  $\partial_h\psi(h; \theta,\jumpp{z},{\rm Fr})=0$.  Let
  $\psi_2({\rm Fr}):=\psi(\widetilde{h}_2(
  \theta,\jumpp{z},{\rm Fr}); \theta,\jumpp{z},{\rm Fr})$.
  One has
  \[
  \partial_{{\rm Fr}^2}\psi_2({\rm Fr})=\displaystyle\frac{1}{\theta \widetilde{h}_2}-1.
  \]
  Since $\partial_{{\rm Fr}^2}\widetilde{h}_2>0$, one can
  draw the variation of $\psi_2({\rm Fr})$ as in the Table
  1. In the case $1/\theta>1-\jumpp{z}$, there is a value
  ${\rm Fr}_{*}$ such that $\psi_2({\rm Fr}_{*})=0$ and
  $\psi_2({\rm Fr})<0$ for any Froude number ${\rm Fr}$
  satisfying $({\rm Fr})^2<({\rm Fr}_{*})^2$.  For such a
  value of the Froude number, there are two solutions that
  are both smaller than $1/\theta$.  Taking into account
  that $\psi(h; \theta, \jumpp{z}, {\rm Fr})>0$, if
  $1-\jumpp{z}<h<1/\theta$, one can qconclude the first
  inequality in (\ref{RP_03}).  Similar arguments can be
  invoked to prove the second inequality in (\ref{RP_03}).
\end{proof}

\begin{lemma}
  \label{lemma-RP-02}
  Let $\theta$ and $\jumpp{z}$ be fixed.  Let ${\rm Fr}$ be
  a variable parameter in the terrain jump equation
  $\psi(h; \theta,\jumpp{z},{\rm Fr})=0$ and let
  ${\rm Fr}_{*}$ and ${\rm Fr}^{*}$ be the critical values
  given by Lemma~{\rm \ref{lemma-RP-01}}. Assume that the
  solutions $h_1<h_2$ of the equation
  $\psi(h,\theta, \jumpp{z},{\rm Fr})=0$ exist.  Then we can
  affirm:

  {\rm (a)} If $1/\theta>1-\jumpp{z}$, then
  \begin{equation}
    \label{RP_04}
    \begin{array}{ll}
      {\cal E}(h_2)<{\cal E}(h_1),&\,\,{\rm if\,\,} {\rm Fr}^2<({\rm Fr}_{*})^2,\\
      {\cal E}(h_1)<{\cal E}(h_2),&\,\,{\rm if\,\,} {\rm Fr}^2>({\rm Fr}^{*})^2.\\
    \end{array}
  \end{equation}

  {\rm (b)} If $1/\theta\leq 1-\jumpp{z}$, then there is a
  value $\widetilde{{\rm Fr}}$ with
  \begin{equation}
    \label{RP_05}
    \widetilde{{\rm Fr}}=\frac{1}{\theta}+\frac{-b(\theta)}{-b(\theta)+\theta}\left(
      1-\jumpp{z}-\frac{1}{\theta}\right),
  \end{equation}
  such that
  \begin{equation}
    \label{RP_06}
    \begin{array}{ll}
      {\cal E}(h_2)<{\cal E}(h_1)&\,\,{\rm if\,\,} {\rm Fr}^2<\widetilde{{\rm Fr}},\\
      {\cal E}(h_1)<{\cal E}(h_2)&\,\,{\rm if\,\,} {\rm Fr}^2>\widetilde{{\rm Fr}},\\
      {\cal E}(h_1)={\cal E}(h_2)&\,\,{\rm if\,\,} {\rm Fr}^2=\widetilde{{\rm Fr}}.
    \end{array}
  \end{equation}
  In addition, let ${\rm Fr}^2_{+}=(u^+)^2/gh^+$ be the
  Froude number defined by physical solution of the terrain
  jump equation. Then,
  \begin{equation}\label{RP_06-01}
    ({\rm Fr}^2-\widetilde{{\rm Fr}})({\rm Fr }_{+}^2-1)>0.
  \end{equation}
\end{lemma} 

\begin{proof}
  The inequalities (\ref{RP_04}) immediately result from the
  definition of ${\cal E}(h)$ and (\ref{RP_03}).  To prove
  the inequalities (\ref{RP_06}), we proceed as follows.  We
  observe that $\partial_{{\rm Fr}^2}{\cal E}(h_1)<0$ and
  $\partial_{{\rm Fr}^2}{\cal E}(h_2)>0$.  Moreover, there
  is a value of ${\rm Fr}$ such that
  ${\cal E}(h_1)={\cal E}(h_2)$.  This equality implies that
  \[
  h_1+h_2=\frac{1}{\theta}+1-\jumpp{z}.
  \]
  The above equality  allows us to calculate the negative solution
  $h_3=1/b(\theta)$ of the equation
  $\psi(h,\theta, \jumpp{z},{\rm Fr})=0$. Then,
  \[
  {\rm Fr}^2\left(\frac{b(\theta)}{\theta}-1\right)+
  \frac{b(\theta)}{\theta}\left(\frac{1}{b(\theta)}-(1-\jumpp{z})\right)=0.
  \]
  This proves (\ref{RP_05}).

  Note that $\widetilde{{\rm Fr}}$ satisfies the following
  inequalities
  \[
  \frac{1}{\theta}\leq\widetilde{{\rm Fr}}\leq 1-\jumpp{z}. 
  \]
  To prove (\ref{RP_06-01}) we observe that if
  ${\rm Fr}^2<\widetilde{{\rm Fr}}$, then (\ref{RP_06})-1 and
  (\ref{RP_02}) imply that the physical solution, $h_2$, is
  greater than $1-\jumpp{z}$. One has
  \[
  {\rm Fr}_{+}^2(\theta;\jumpp{z};{\rm Fr}):=\frac{{\rm Fr}^2}{\theta^2 h_2^3}<
  \frac{\widetilde{{\rm Fr}}}{\theta^2 h_2^3}<\frac{ h_2}{\theta^2 h_2^3}<
  \frac{ 1}{\theta^2 h_2^2}<1.
  \]  
  In the case ${\rm Fr}^2>\widetilde{{\rm Fr}}$, then
  (\ref{RP_06})-2 and (\ref{RP_02}) imply that the physical
  solution, $h_1$, is less than $1/\theta$ and inequality
  (\ref{RP_06-01}) can be proven in a similar way.
\end{proof}

\subsection{Dam-break problem}

A very interesting Riemann Problem is the dam-break.  Here,
we consider an extended dam-break problem in which we also
have a soil surface jump.  It can be formulated as
\begin{equation}
  \label{db_eq.01}
  \begin{array}{l}
    u_L=0,\quad u_R=0;\\
    h_L+z_L>h_R+z_R. 
  \end{array}
\end{equation}  

We seek a solution of the problem that is defined by a
composite wave.  In the presence of a jump in one of terrain
function, the composite wave must include a 3-wave since it
is the only wave that supports a jump in a terrain function.
Also, in the case of dam-break problem, the composite wave
must include a 1-rarefaction wave issuing from the left
state $U_L$ and ending at a point $U$ with
$\lambda_1(U)\leq 0$ and this $U$ must be an admissible
state for a 3-wave. It follows that it is essential to
investigate the composite wave $W_3(W_1(h;h_L,u_L)$, where
$W_1(h;h_L,u_L)$ is restricted to the segment
$\lambda_1(W_1(h;h_L,u_L))\leq 0$, for $h<h_L$.

For this purpose, we consider that the terrain functions
data $\theta_R$, $\theta_L$, $z_R$ and $z_L$ are frozen
and we study what happens with the solution when the
hydrodynamic data $h_L$ and $h_R$ take different values.

\def\rtheta{\theta} 
We denote by $h_L^{\#}$ the value of $h$
where the rarefaction 1-wave issuing from $(h_L,u_L)$ with
${\rm Fr}_L^2<1$ intersects the curve ${\rm Fr}(u,h)=1$
\begin{equation}
  \label{RP-notation.1}
  h_L^{\#}=h_L\frac{({\rm Fr}_L+2)^2}{9},\quad 
  u_L^{\#}=\sqrt{{\rm g}h_L^{\#}}.
\end{equation}
Here, we use the notations $\jumpp{z}=z_R-z_L$ and
$\rtheta=\theta_R/\theta_L$.  As in Lemma~\ref{lemma-RP-02},
we introduce
\begin{equation}
  \label{RP-notation.2}
  \widetilde{{\rm Fr}}(h) = 
  \frac{1}{\rtheta}+
  \frac{-b(\rtheta)}{-b(\rtheta)+\rtheta}
  \left(
    1-\frac{|z|}{h}-\frac{1}{\rtheta}
  \right).
\end{equation}

The curve $W_3(W_1(h;h_L,u_L))$ can be defined for $h$ close
enough to $h_L$ (${\rm Fr}_L=0$!) but it is questionable
whether it can be defined for any $h_L^{\#}<h<h_L$.

The next proposition provides sufficient conditions on $h_L$
to guaranty that the curve $W_3(W_1(h;h_L,u_L))$ is well
defined for any $h_L^{\#}<h<h_L$ and also describes
properties of this curve.  It is based on the fact that
there are some circumstances that allow to solve the terrain
jump equation for any Froude number, see
Lemma~\ref{lemma-RP-01}.

\begin{proposition}
  \label{RP-prop.1}
  Suppose that:\\
  \underline{{\rm Case a:}} $\theta_R>\theta_L$, $\jump{z}>0$ and
  \begin{equation}
    \label{RP.07}
    h_L^{\#}>(z_R-z_L)\frac{\theta_R}{\theta_R-\theta_L}.
  \end{equation}
  \underline{{\rm Case b:}} $\theta_R>\theta_L$, $\jump{z}<0$,
  (without restrictions on $h_L$).\\
  \underline{{\rm Case c:}} $\theta_R<\theta_L$, $\jump{z}<0$
  and
  \begin{equation}
    \label{RP.08}
    h_L<(z_R-z_L)\frac{\theta_R}{\theta_R-\theta_L}.
  \end{equation}

  Then:\\
  {\rm (1)} In all three cases, one can define the curve
  3--wave $W_3(W_1(h; h_L,u_L))$, for
  $h_L^{\#}\leq h \leq h_L$.\\
  {\rm (2)} $W_3(W_1(h; h_L,u_L))$, for
  $h_L^{\#}\leq h \leq h_L$ is a disconnected curve composed
  by two continuous branches, one with Froude number greater
  than one and the other one with Froude number smaller than
  one, in the {\rm Case a} and in the {\rm Case b} provided
  that
  \begin{equation}
    \label{RP.09}
    \widetilde{\rm Fr}(h_L^{\#})<1.
  \end{equation}
  {\rm (3)} $W_3(W_1(h; h_L,u_L))$, for
  $h_L^{\#}\leq h \leq h_L$ is a continuous curve with
  Froude number smaller than one in {\rm Case c} and in {\rm
    Case b} provided that
  \begin{equation}
    \label{RP.10}
    \widetilde{\rm Fr}(h_L^{\#})>1.
  \end{equation}
\end{proposition}

\begin{proof} 
  (1). If the conditions (\ref{RP.07}) and (\ref{RP.08}) are
  satisfied, then one can use Lemma~\ref{lemma-RP-01}--a.
  To prove (2) and (3), we use Lemma~\ref{lemma-RP-02}--b,
  the estimation for $\widetilde{{\rm Fr}}(h_L^{\#})$,
  \[
  \widetilde{{\rm Fr}}(h_L^{\#}) <1-\frac{\jump{z}}{h_L^{\#}}<1,
  \]
  in Case a and 
  \[
  \widetilde{{\rm Fr}}(h_L^{\#}) >\frac{1}{\rtheta}>1,
  \]
  in Case c, combined with the property that the functions
  $\widetilde{{\rm Fr}}(h)$ and ${\rm Fr}(h)^2$ have at most
  only one intersection point on the interval
  $[h_L^{\#},h_L]$.
\end{proof}

The remaining case, $\theta_R<\theta_L$ and $\jump{z}<0$, is
less complicated than the ones analyzed in
Proposition~\ref{RP-prop.1} and $W_3(W_1(h; h_L,u_L))$ can
be completely described.

\begin{proposition}
  \label{RP-prop.2}
  Suppose that $\theta_R<\theta_L$ and $\jump{z}>0$. There
  is a critical value $h^\#< h_c\leq h_L$ such that the
  following properties hold:\\
  {\rm (1)} The curve 3--wave $W_3(W_1(h; h_L,u_L))$ can be
  defined for $h_c\leq h \leq h_L$ and it is connected.\\
  {\rm (2)} Moreover,
  \begin{equation}
    \label{RP_10}
    Fr_+(W_3(W_1(h; h_L,u_L)))\leq Fr_+(W_3(W_1(h_c;
    h_L,u_L))), \quad h_c\leq h \leq h_L,
  \end{equation}
  where $Fr_+ := \theta^{-1} \beta^{-3/2} Fr_-$.\\
  {\rm (3)} The curve 3--wave $W_3(W_1(h; h_L,u_L))$ does
  not exist for $h^\#\leq h <h_c$.
\end{proposition}

\begin{proof}
  {\rm (1)} In the case $1/\theta>1>1-\jumpp{z}$, we apply
  Lemma \ref{lemma-RP-01}--a to show that
  \[
  {\rm Fr}_{*}<1<{\rm Fr}^{*}.
  \]
  Let $\eta= 1/\theta+(a(\theta)+2)(1/\theta-(1-\jumpp{z}))$
  be the value of ${\rm Fr}^2$, where the positive root
  $\widetilde{h}_2({\rm Fr})$ of the derivative of
  $\partial \psi(h; \theta,\jumpp{z},{\rm Fr})/\partial y$
  equals to $1/\theta$.  From $a(\theta)+2>0$ results that
  $\eta>1$, hence ${\rm Fr}^{*}>1$ and
  $\widetilde{h}_2(1)<1/\theta$.  To show that
  ${\rm Fr}_{*}<1$, it is sufficient to prove that the
  minimum $\psi_2(1)$ of $\psi(y;\theta;\jumpp{z}, 1)$ is
  positive.  One has
  \[
  \psi_2(1)=\frac{1}{\theta}-\widetilde{h}_2(1)-\widetilde{h}_2(1)
  \left(\widetilde{h}_2(1)-(1-\jumpp{z})\right)\left(a(\theta)+
    b(\theta)\widetilde{h}_2(1)\right)
  \]
  Since $\widetilde{h}_2(1)<1/\theta$, $a(\theta)<0$ and
  $b(\theta)<0$, it is sufficient that
  $\widetilde{h}_2(1)>(1-\jumpp{z})$ to prove the inequality
  $\psi_2(1)>0$.

  Taking into account that $\widetilde{h}_2(1)$ is the
  greatest root of the quadratic equation
  $\partial_y\psi(y;\theta,\jumpp{z},1)=0$, the inequality
  holds if
  $\partial_y\psi(1-\jumpp{z};\theta,\jumpp{z},1)<0$.  One
  has
  \[
  \partial_y\psi(1-\jumpp{z};\theta,\jumpp{z},1)=
  -b(\theta)\left(1-\jumpp{z}\right)^2-a(\theta)
  \left(1-\jumpp{z}\right)-1.
  \]
  The function
  \[
  f(y):=-b(\theta)y^2-a(\theta) y-1=0
  \]
  is monotone increasing on the interval $(0,\infty)$.
  Since $0<1-\jumpp{z}<1$, then
  \[
  f(0)<f(1-\jumpp{z})<f(1).
  \]
  Conequently,
  \[
  -1<f(1-\jumpp{z})<-b(\theta)-a(\theta)-1.
  \]
  As
  $-b(\theta)-a(\theta)-1=
  -b(\theta)+b(\theta)/\theta=b(\theta)(1/\theta-1)<0$,
  it results that
  \[
  f(1-\jumpp{z})<0.
  \]
  So, the root $\widetilde{h}_2(1)>1-\jumpp{z}$.\medskip

  {\rm (2)} The inequality (\ref{RP_10}) can be proved as
  follows.  The obtain a point $W_3(W_1(h;h_L;u_L))$, one
  must solve the terrain equation when the left state is the
  point $W_1(h;h_L;u_L)$.  Let $h_{c}$ be the minimum value
  of $h$ on the interval $(h^{\#}_L, h_L)$ for which the
  terrain equation is solvable having the point
  $W_1(h;h_L;u_L)$ as left state.  We show that on the curve
  $W_3(W_1(h;h_L;u_L))$ the Froude number is a monotone
  decreasing function of $h$.  Let $W_1(h;h_L;u_L)$ be a
  left state in the terrain equation, ${\rm Fr}_{-}(h)$ be
  the Froude number of the left state and let $\beta(h)$ be
  the solution of the equation.  Using $\beta(h)$, we define
  the 3-wave $W_3(W_1(h;h_L;u_L))$.  Let ${\rm Fr}_{+}(h)$
  be the Froude number of $W_3(W_1(h;h_L;u_L))$.  One has
  \[
  {\rm Fr}_{+}(h)=\frac{{\rm Fr}_{-}(h)}{\theta
    \beta(h)^{3/2}}
  \]
  and then,
  \[
  \partial_h{\rm Fr}_{+}(h)=\frac{\partial_h{\rm
      Fr}_{-}(h)}{\theta \beta(h)^{3/2}}- 3/2\frac{{\rm
      Fr}_{-}(h)}{\theta \beta(h)^{5/2}}\partial_h\beta(h).
  \]
  It is relatively easy to show that on the curve
  $W_1(h;h_L;u_L)$, one has $\partial_h{\rm Fr}_{-}(h)<0$.

  We show that $\partial_h\beta(h)$ is a negative function,
  too.  Using a standard implicit function theorem, we can
  write
  \begin{equation*}
    \begin{split}
      0=\psi(\beta(h);\theta,\jumpp{z}(h),{\rm Fr}_{-}(h))&=
      \partial_y\psi(y;\theta,\jumpp{z}(h),{\rm
        Fr}_{-}(h))\Big|_{y=\beta(h)}\cdot\partial_h\beta(h)+\\
      &+\phi({\beta(h);\theta,\jumpp{z}(h),{\rm
          Fr}_{-}(h)}),
    \end{split}
  \end{equation*}
where we used the notation
\begin{equation*}
  \begin{split}
    \phi(\beta(h);\theta,\jumpp{z}(h),{\rm Fr}_{-}(h)):=&\\
    +&\partial_{\jumpp{z}}\psi(\beta(h);\theta,\jumpp{z}(h),{\rm Fr}_{-}(h))\cdot\partial_h\jumpp{z}(h)+\\
    +&\partial_{{\rm
        Fr}_{-}}\psi(\beta(h);\theta,\jumpp{z}(h),{\rm
      Fr}_{-}(h))\cdot\partial_h{{\rm Fr}_{-}(h)}.
  \end{split}
\end{equation*}
We have,
\begin{equation*}
  \begin{array}{l}
    \phi(\beta(h);\theta,\jumpp{z}(h),{\rm Fr}_{-}(h))=\\\\
    =\displaystyle 
    \frac{\jumpp{z}(h)}{h}\beta(h)(a(\theta)+b(\theta)\beta(h))+
    2{\rm Fr}_{-}(h)
    (1/\theta-\beta(h))\partial_h{\rm Fr}_{-}(h)<0.
  \end{array}
\end{equation*}
For the last inequality, we used that $a(\theta)$ and
$b(\theta)$ are negative functions, $1/\theta-\beta(h)>0$
(see Lemma \ref{lemma-RP-02}--a) and
$\partial_h{\rm Fr}_{-}(h)<0$.  Using a standard implicit
function theorem, we can write
\[
\partial_h\beta(h)=-\partial_h\psi(\beta(h);\theta,\jumpp{z}(h),{\rm
  Fr}_{-}(h))/\psi(\beta(h);\theta,\jumpp{z}(h),{\rm Fr}_{-}(h))>0,
\]
consequently,
\[
\partial_h{\rm Fr}_{+}(h)<0.
\]
\end{proof}

In all cases that are not covered by the
Proposition~\ref{RP-prop.1}, we can not say in advance if
the point $U_L^{\#}$ is an admissible state for $W_3$ or if
its Froude number is greater or smaller than one.

The general strategy to solve the dam-break problem is to
find out if the backward wave $W^B_2(U_R)$ intersects a
segment of the composite wave $W_3(W_1(h;U_L))$.  In the
case of negative answer, we try to interpose a 1--wave
between $W_3(W_1)$ and $W_2$.

We will show that there are three different structures of
the composite waves that can solve the dam-break problem.
All three algorithms are effective for the more general
Riemann Problem (than the dam-break problem) but restricted
to ${\rm Fr}_L^2<1$ and ${\rm Fr}_R^2<1$.

\subsubsection{Constructive Algorithms}
\begin{samepage}
  {\bf Riemann Constructor. Type I}
  \begin{enumerate}
  \item[{\bf Step 1}] {\it Build the curve}
    \[
    W_3(W_1(h;h_L,u_L);\theta,\jump{z}).
    \]
  \item[{\bf Step 2}] {\it Find the intersection point $h_*$
      such that}
    \[
    W_2^B(h_*;h_R,u_R)=W_3(W_1(h_*;h_L,u_L);\theta,\jump{z}).
    \]
  \item[{\bf Step 3}] {\it The composite wave curve of the
      solution for the Riemann problem is}
    \[
    W_2(h_R;W_3(W_1(h_*;h_L,u_L);\theta,\jump{z})).
    \]
  \end{enumerate}
\end{samepage}
\begin{figure}[!ht]
  \centering
  \begin{tabular}{cc}
    \multicolumn{2}{c}
    {
    \begin{tabular}{cc}
      &\small{$\mathbf{\jump{z}=-0.2\quad\theta_R/\theta_L=0.5}$}\\
      \begin{turn}{90}
        \hspace{20mm}\small{${u \;  [m/s]}$}
      \end{turn}
      &\includegraphics[width=0.5\linewidth]{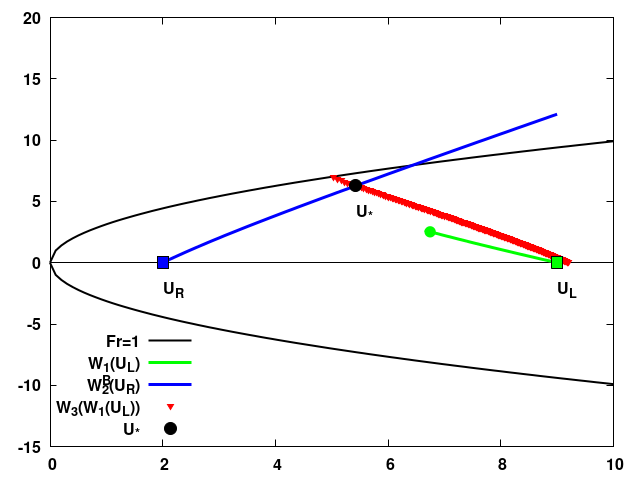}\\
      &\small{${h \; [m]}$}
    \end{tabular}
    }\\
    \begin{tabular}{cc}
      &\small{{\bf Free water surface, $\mathbf {t=0.7s}$}}\\
      \begin{turn}{90}
        \hspace{20mm}\small{${h \; [m]}$}
      \end{turn}
      &\includegraphics[width=0.4\linewidth]{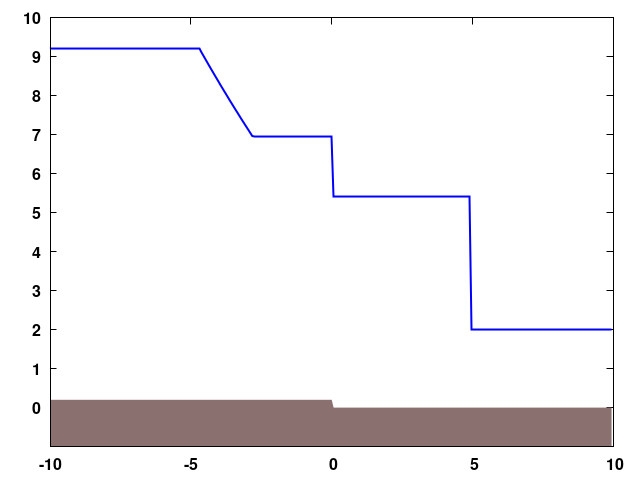}\\
      &\small{${x \; [m]}$}
    \end{tabular}
    &
    \begin{tabular}{cc}
      &\small{{\bf Water velocity, $\mathbf{t=0.7s}$}}\\
      \begin{turn}{90}
        \hspace{20mm}\small{${u \; [m/s]}$}
      \end{turn}
      &\includegraphics[width=0.4\linewidth]{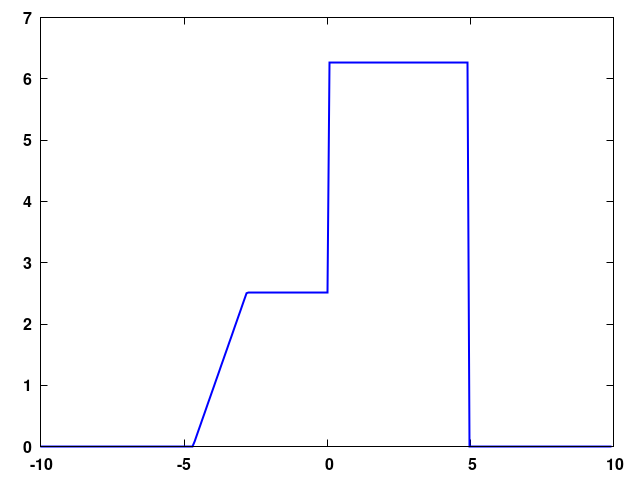}\\
      &\small{${x \; [m]}$}
    \end{tabular}\\
  \end{tabular}
  \caption{Type~I solution of Riemann Problem: the phase
    portrait (on top), the water level (left bottom) and the
    water velocity (right bottom).}
  \label{fig_typeI}
\end{figure}
In Step 1, we build the maximal connected curve that
includes $h=h_L$.  It really exists if ${\rm Fr}_L=0$ or if
the jumps in terrain functions are small enough.  The
algorithm furnishes a solution only if the intersection
point $h_{*}$ searched in {\it Step 2} exists, otherwise
this algorithm does not provide a solution.  A typical
pattern of solution is illustrated in
Figure~\ref{fig_typeI}.
\medskip

{\bf Riemann Constructor. Type II}
\begin{samepage}
  \begin{enumerate}
  \item[{\bf Step 1}] {\it Build the curve}
    \[
    W_3(W_1(h;h_L,u_L);\theta,\jump{z}).
    \]
  \item[{\bf Step 2}] {\it Find the intersection point $U_2$
      and $h_1$ such that}
    \[
    U_2=W_3(W_1(h_1;h_L,u_L);\theta,\jump{z}),\quad {\rm
      Fr}(U_2)=1.
    \]
  \item[{\bf Step 3}] {\it Find the intersection point
      $h_*<h_2$ such that}
    \[
    W_2^B(h_*;h_R,u_R)=W_1(h_*;W_3(W_1(h_1;h_L,u_L);\theta,\jump{z})).
    \]
  \item[{\bf Step 4}] {\it The composite wave curve of the
      solution for the Riemann Problem is}
    \[
    W_2(h_R;W_1(h_*;W_3(W_1(h_*;h_L,u_L);\theta,\jump{z}))).
    \]
  \end{enumerate}
\end{samepage}
This algorithm works only if there is a point on the curve
$W_3(W_1(\cdot;\cdot)$ with Froude number equal to one.  A
solution of this type exits if the point
$W_3(W_1(h_1;U_L);\theta,\jump{z})$ is below the curve
$W_2^B(h;h_R,u_R)$, for $h<h_R$.  A typical pattern of
solution is illustrated in Figure~\ref{fig_typeII}.
\begin{figure}[t]
  \centering
  \begin{tabular}{cc}
    \multicolumn{2}{c}
    {
    \begin{tabular}{cc}
      &\small{$\mathbf{\jump{z}=-0.2\quad\theta_R/\theta_L=0.5}$}\\
      \begin{turn}{90}
        \hspace{20mm}\small{${u \;  [m/s]}$}
      \end{turn}
      &\includegraphics[width=0.5\linewidth]{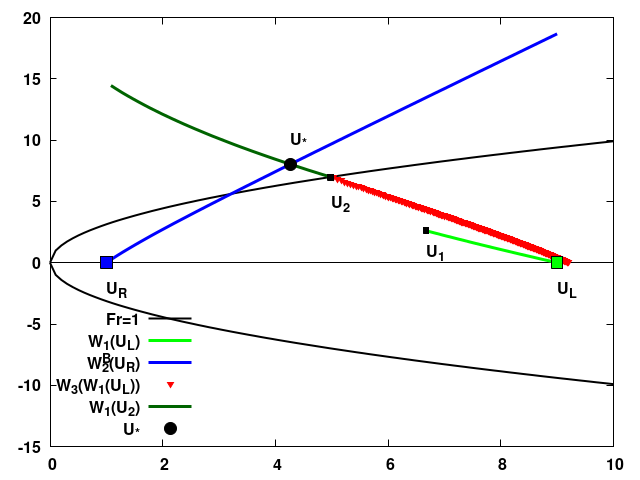}\\
      &\small{${h \; [m]}$}
    \end{tabular}
    }\\
    \begin{tabular}{cc}
      &\small{{\bf Free water surface, $\mathbf {t=0.7s}$}}\\
      \begin{turn}{90}
        \hspace{20mm}\small{${h \; [m]}$}
      \end{turn}
      &\includegraphics[width=0.4\linewidth]{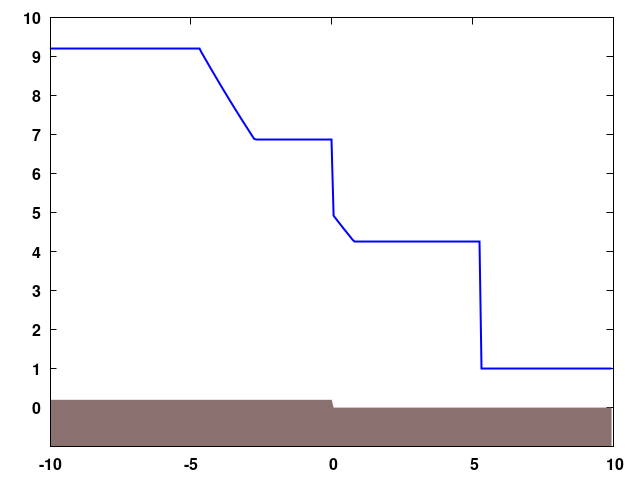}\\
      &\small{${x \; [m]}$}
    \end{tabular}
    &
    \begin{tabular}{cc}
      &\small{{\bf Water velocity, $\mathbf{t=0.7s}$}}\\
      \begin{turn}{90}
        \hspace{20mm}\small{${u \; [m/s]}$}
      \end{turn}
      &\includegraphics[width=0.4\linewidth]{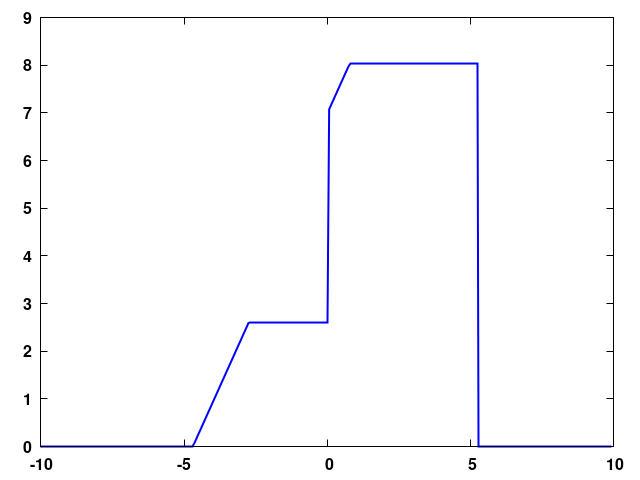}\\
      &\small{${x \; [m]}$}
    \end{tabular}\\
  \end{tabular}
  \caption{Type~II solution of Riemann Problem: the phase
    portrait (on top), the water level (left bottom) and the
    water velocity (right bottom).}
  \label{fig_typeII}
\end{figure}
\medskip

{\bf Riemann Constructor. Type III}
\begin{samepage}
  \begin{enumerate}
  \item[{\bf Step 1}] {\it Find the intersection point $h_1$
      such that}
    \[
    {\rm Fr}(W_1(h_1;h_L,u_L))=1.
    \]
  \item[{\bf Step 2}] {\it Find the point $(h_2,u_2)$ such
      that}
    \[
    W_3(W_1(h_1;h_L,u_L);\theta,\jump{z}))=(h_2,u_2).
    \]
  \item[{\bf Step 3}] {\it Find the intersection point $h_3$
      such that}
    \[
    W_2^B(h_3;h_R,u_R)=W_1(h_3;h_2,u_2).
    \]
  \item[{\bf Step 4}] {\it The composite wave curve of the
      solution for the Riemann problem is}
    \[
    W_2(h_R;W_1(h_3;W_3(W_1(h_1;h_L,u_L);\theta,\jump{z}))).
    \]
  \end{enumerate}
\end{samepage}
The algorithm can be used only if $U_2$ exists and
${\rm Fr}(U_2)>1$.  The Proposition~\ref{RP-prop.1} offers
results to guaranty that the algorithm of Type~III makes
sense.  To have a solution, one needs to verify that the
point $h_3$ exists and if it is also verified the inequality
$h_3>h_2$, then the shock speed $\sigma_1(h_3; h_2,u_2$ must
be positive.  A typical pattern solution is illustrated in
Figure~\ref{fig_typeIII}.
\begin{figure}[t]
  \centering
  \begin{tabular}{cc}
    \multicolumn{2}{c}
    {
    \begin{tabular}{cc}
      &\small{$\mathbf{\jump{z}=0.2\quad\theta_R/\theta_L=2}$}\\
      \begin{turn}{90}
        \hspace{20mm}\small{${u \;  [m/s]}$}
      \end{turn}
      &\includegraphics[width=0.5\linewidth]{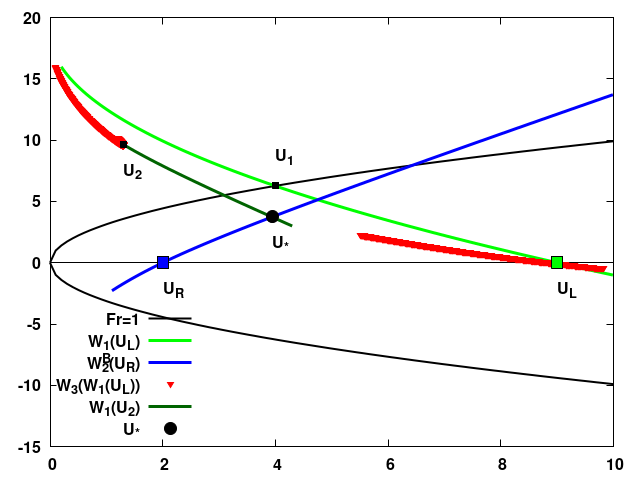}\\
      &\small{${h \; [m]}$}
    \end{tabular}
    }\\
    \begin{tabular}{cc}
      &\small{{\bf Free water surface, $\mathbf {t=0.7s}$}}\\
      \begin{turn}{90}
        \hspace{20mm}\small{${h \; [m]}$}
      \end{turn}
      &\includegraphics[width=0.4\linewidth]{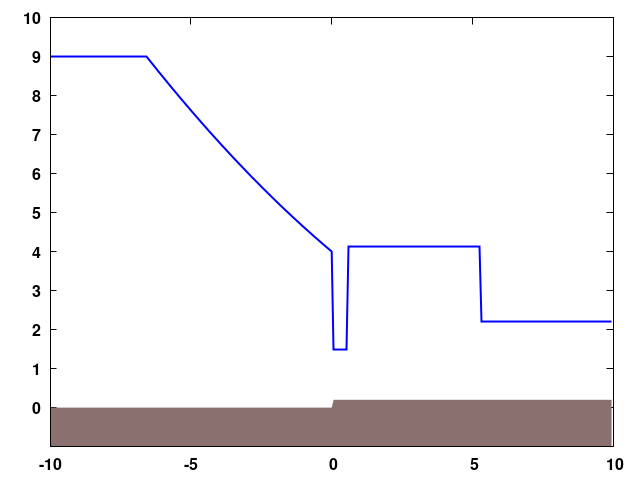}\\
      &\small{${x \; [m]}$}
    \end{tabular}
    &
    \begin{tabular}{cc}
      &\small{{\bf Water velocity, $\mathbf{t=0.7s}$}}\\
      \begin{turn}{90}
        \hspace{20mm}\small{${u \; [m/s]}$}
      \end{turn}
      &\includegraphics[width=0.4\linewidth]{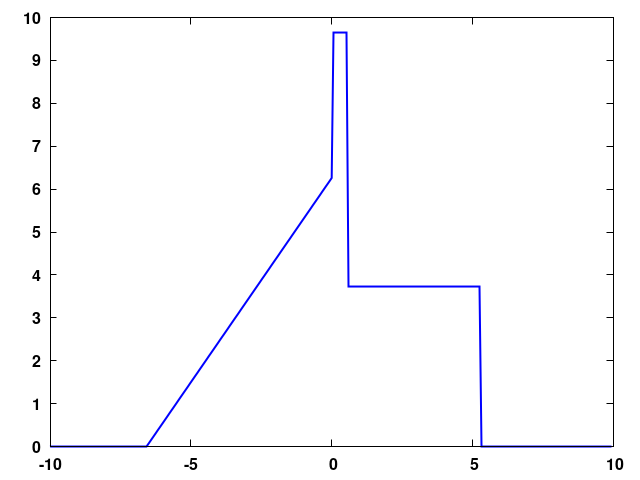}\\
      &\small{${x \; [m]}$}
    \end{tabular}\\
  \end{tabular}
  \caption{Type~III solution of Riemann Problem: the phase
    portrait (on top), the water level (left bottom) and the
    water velocity (right bottom).}
  \label{fig_typeIII}
\end{figure}

\subsubsection{On the existence of the solution of the
  dam-break problem}
We will provide some sufficient conditions to have a
solution by composite waves of the dam-break problem.  We
will also indicate the cases where the solution by composite
wave is not possible.  We assume that the terrain data
$\{z_L, z_R, \theta_L, \theta_R\}$ and $h_L$ are frozen.  We
will use the notation $h_{\rm max}= h_L+z_L-z_R$, $h_L^{\#}$
and $\widetilde{{\rm Fr}}(h)$ given by (\ref{RP-notation.1})
and (\ref{RP-notation.2}), respectively, the critical number
$h_c$ whose existence was proved in proposition
\ref{RP-prop.2}.

The following theorem gathers together the results
formulated in Propositions \ref{RP-prop.1}, \ref{RP-prop.2}
and in the Constructive Algorithms (of type I, II, III) of
the solution for the dam-break problem.

\begin{theorem}
  Let $\{z_L, z_R, \theta_L, \theta_R\}$
  and $h_L$ be given.  We assume that:\\
  \underline{{\rm Case a:}}
  \begin{equation}
    \label{RP.07-1}
    \theta_R>\theta_L, \quad z_R>z_L,\quad h_L+z_L>z_R, \quad
    h_L^{\#}>(z_R-z_L)\displaystyle \frac{\theta_R}{\theta_R-\theta_L}.
  \end{equation}
  There is a value $\widetilde{h}$ with
  $h_L^{\#}<\widetilde{h}<h_L$ such that one can build the
  curve $W_3(W_1(h;h_L,0))$ for
  $\widetilde{h}\leq h \leq h_L$, the point
  $U_2=W_3(U_L^{\#})$ and the 1-wave $W_1(h;U_2)$ for
  $h<h_0$, ($\sigma_1(h_0;U_2)=0$).  In this case, there are
  two values $\xi^a_1<\xi^a_2<h_{\rm max}$ satisfying the
  properties:
  \begin{enumerate}
  \item For $h_R<\xi^a_1$, $W^B_2(h_R,0)$ intersects
    $W_1(h;U_2)$ and the solution is given by algorithm III;
  \item For $\xi^a_1<h_R<\xi^a_2$, there is no solution;
  \item For $\xi^a_2<h_R<h_{\rm max}$, $W^B_2(h_R,0)$
    intersects $W_3(W_1(h;h_L,0))$ and the solution is given
    by algorithm I.
  \end{enumerate}
  \underline{{\rm Case b:}}
  \begin{equation}
    \label{RP.08-1}
    \theta_R>\theta_L,\quad z_R<z_L.
  \end{equation}
  Depending on the value of $\widetilde{\rm
    Fr}(h_L^{\#})$, one has:\\
  \underline{{\rm Case b1:}}
  \begin{equation}
    \label{RP.08-2}
    \widetilde{\rm Fr}(h_L^{\#})<1.
  \end{equation}
  There is a value $\widetilde{h}$ with
  $h_L^{\#}<\widetilde{h}<h_L$ such that one can build the
  curve $W_3(W_1(h;h_L,0))$ for
  $\widetilde{h}\leq h \leq h_L$, the point
  $U_2=W_3(U_L^{\#})$ and the 1-wave $W_1(h;U_2)$ for
  $h<h_0$, ($\sigma_1(h_0;U_2)=0$).  In this case, there are
  two values $\xi^{b1}_1<\xi^{b1}_2<h_{\rm max}$ satisfying
  the properties:
  \begin{enumerate}
  \item For $h_R<\xi^{b1}_1$, $W^B_2(h_R,0)$ intersects
    $W_1(h;U_2)$ and the solution is given by algorithm III;
  \item For $\xi^{b1}_1<h_R<\xi^{b1}_2$, there is no
    solution;
  \item For $\xi^{b1}_2<h_R<h_{\rm max}$, $W^B_2(h_R,0)$
    intersects $W_3(W_1(h;h_L,0))$ and the solution is given
    by algorithm I.
  \end{enumerate}
  \underline{{\rm Case b2:}}
  \begin{equation}
    \label{RP.08-3}
    \widetilde{\rm Fr}(h_L^{\#})>1.
  \end{equation}
  One can build the curve $W_3(W_1(h;h_L,0))$ for
  $h_L^{\#}\leq h \leq h_L$. In this case, there is one
  value $\xi^{b2}<h_{\rm max}$ with the following
  properties:
  \begin{enumerate}
  \item For $h_R<\xi^{b2}$ there is no solution;
  \item For $\xi^{b2} \leq h_R<h_{\rm max}$, $W^B_2(h_R,0)$
    intersects $W_3(W_1(h;h_L,0))$ and the solution is given
    by algorithm I.
  \end{enumerate}
  \underline{{\rm Case c:}}
  \begin{equation}
    \label{RP.09-1}
    \theta_R<\theta_L, \quad z_R<z_L, \quad
    h_L<(z_R-z_L)\displaystyle
    \frac{\theta_R}{\theta_R-\theta_L}.
  \end{equation}
  One can build the curve $W_3(W_1(h;h_L,0))$ for
  $h_L^{\#}\leq h \leq h_L$.  In this case, there is one
  value $\xi^c<h_{\rm max}$ with the following properties:
  \begin{enumerate}
  \item For $h_R<\xi^c$ there is no solution;
  \item For $\xi^c \leq h_R<h_{\rm max}$, $W^B_2(h_R,0)$
    intersects $W_3(W_1(h;h_L,0))$ and the solution is given
    by algorithm I.
  \end{enumerate}
  \underline{{\rm Case d:}}
  \begin{equation}
    \label{RP.10-1}
    \theta_R<\theta_L, \quad z_R>z_L.
  \end{equation}
  There is a value $h_c$ with $h_L^{\#}<h_c<h_L$ such that
  one can build the curve $W_3(W_1(h;h_L,0))$ for
  $h_c\leq h \leq h_L$.  Depending of the value of
  $Fr_+(W_3(W_1(h_c; h_L,u_L)))$, one has:\\
  \underline{{\rm Case d1:}}
  \begin{equation}
    \label{RP.10-2}
    Fr_+(W_3(W_1(h_c; h_L,u_L)))<1
  \end{equation}
  In this case, there is one value $\xi^{d1}<h_{\rm max}$
  with properties:
  \begin{enumerate}
  \item For $h_R<\xi^{d1}$ there is no solution;
  \item For $\xi^{d1} \leq h_R <h_{\rm max}$, $W^B_2(h_R,0)$
    intersects $W_3(W_1(h;h_L,0))$ and the solution is given
    by algorithm I;
  \end{enumerate}
  \underline{{\rm Case d2:}}
  \begin{equation}
    \label{RP.10-3}
    Fr_+(W_3(W_1(h_c; h_L,u_L)))>1
  \end{equation}
  There is a value $h_c$, where $h_c<\widetilde{h}<h_L$,
  with $Fr_+(W_3(W_1(\widetilde{h}; h_L,u_L)))=1$ and one
  builds the curve $W_3(W_1(h;h_L,0))$ for
  $\widetilde{h}\leq h \leq h_L$, the point
  $U_1=W_3(W_1(\widetilde{h};h_L,0) )$ and the 1-wave
  $W_1(h;U_1)$ for $h<h_1$.  In this case, there is
  $\xi^{d2}<h_{\rm max}$ satisfying the properties:
  \begin{enumerate}
  \item For $h_R<\xi^{d2}$, $W^B_2(h_R,0)$ intersects
    $W_1(h;U_1))$ and the solution is given by algorithm II;
  \item For $\xi^{d2}\leq h \leq h_{\rm max}$,
    $W^B_2(h_R,0)$ intersects $W_3(W_1(h;h_L,0))$ and the
    solution is given by algorithm I.
  \end{enumerate}
\end{theorem}

\section{Conclusion}
In this paper, we provide certain analytical solutions of
Riemann Problem for shallow water equations with topography
and vegetation when the jump of the initial data is
arbitrary large.  We introduced three algorithms that allow
to solve the dam-break problem.  Each algorithm provides a
set of elementary waves that are combined to obtain the
solution to the dam-break problem.  This algorithmic
procedure can be extended to solve the general Riemann
Problem.

\end{document}